\documentclass[12pt,twoside]{amsart}
\usepackage{amssymb}

\usepackage[all]{xy}
\nonstopmode

\numberwithin{equation}{section} 

\font\tengothic=eufm10 scaled\magstep 1 \font\sevengothic=eufm7
scaled\magstep 1
\newfam\gothicfam
       \textfont\gothicfam=\tengothic
       \scriptfont\gothicfam=\sevengothic
\def\goth#1{{\fam\gothicfam #1}}

\newtheorem{theorem}{Theorem}[section]

\newtheorem{proposition}[theorem]{Proposition}
\newtheorem{corollary}[theorem]{Corollary}
\newtheorem{conjecture}[theorem]{Conjecture}

\theoremstyle{definition}
\newtheorem{remark}[theorem]{Remark}
\newtheorem{example}[theorem]{Example}

\newcommand{\codim}{\operatorname{codim}}
\newcommand{\coker}{\operatorname{coker}}

\newcommand{\Ann}{\operatorname{Ann}}

\newcommand{\Hom}{\operatorname{Hom}}

\newcommand{\Ext}{\operatorname{Ext}}

\newcommand{\depth}{\operatorname{depth}}

\newcommand{\Hi}{\operatorname{Hilb}}

\newcommand{\Proj}{\operatorname{Proj}}

\newcommand{\proj}[1]
{ \mathchoice
            { {\mathbb P}^{#1} }
            { {\mathbb P}^{#1} }
            { {\mathbb P}^{#1} }
            { {\mathbb P}^{#1} }
          }

\newcommand{\cD}{{\mathcal D}}
\newcommand{\cA}{{\mathcal A}}
\newcommand{\cB}{{\mathcal B}}
\newcommand{\cH}{{\mathcal H}}
\newcommand{\cI}{{\mathcal I}}

\newcommand{\cM}{{\mathcal M}}

\newcommand{\cC}{{\mathcal C}}

\newcommand{\cF}{{\mathcal F}}
\newcommand{\cG}{{\mathcal G}}
\newcommand{\cO}{{\mathcal O}}

\newcommand{\cN}{{\mathcal N}}

\newcommand {\PP}{\mathbb{P}}

\newcommand {\ra}{\longrightarrow}

\begin{document}
\title[Families of determinantal schemes]{Families of determinantal schemes}

\author[Jan O.\ Kleppe, Rosa M.\ Mir\'o-Roig]{Jan O.\ Kleppe, Rosa M.\ Mir\'o-Roig$^{*}$} 
\address{Faculty of Engineering,
         Oslo University College,
         Pb. 4 St. Olavs plass, N-0130 Oslo,
         Norway}
\email{JanOddvar.Kleppe@iu.hio.no}
\address{Facultat de Matem\`atiques,
Departament d'Algebra i Geometria, Gran Via de les Corts Catalanes
585, 08007 Barcelona, SPAIN } \email{miro@ub.edu}

\date{\today}

\thanks{$^*$ Partially supported by MTM2010-15256 .}


\begin{abstract}  Given integers $a_0\le a_1\le \cdots  \le a_{t+c-2}$ and
$b_1\le \cdots \le b_t$, we denote by
$W(\underline{b};\underline{a})\subset \Hi ^p(\PP^{n})$ the locus
of good  determinantal schemes $X\subset \PP^{n}$ of codimension
$c$ defined by the maximal minors of a $t\times (t+c-1)$
homogeneous matrix with entries homogeneous polynomials of degree
$a_j-b_i$. The goal of this paper is to extend and complete
the results given by the authors in \cite{KM} and determine under
weakened  numerical assumptions the dimension of
$W(\underline{b};\underline{a})$ as well as whether the closure of
$W(\underline{b};\underline{a})$  is a generically smooth
irreducible component of $\Hi ^p(\PP^{n})$.
 \end{abstract}


\maketitle



\section{Introduction} \label{intro}

In this paper, we will deal with good and standard determinantal schemes. A
scheme $X\subset \PP^{n}$ of codimension $c$ is called {\em standard
  determinantal} if its homogeneous saturated ideal can be generated by the
maximal minors of a homogeneous $t \times (t+c-1)$ matrix, and $X$ is said to
be {\em good determinantal} if it is standard determinantal and a generic
complete intersection. We denote the Hilbert scheme by $\Hi ^p(\PP^n)$. Given
integers $a_0\le a_1\le \cdots \le a_{t+c-2}$ and $b_1\le \cdots \le b_t$, we
denote by $W(\underline{b};\underline{a})\subset \Hi ^p(\PP^{n})$ (resp.
$W_s(\underline{b};\underline{a})$) the locus of good (resp. standard)
determinantal schemes $X\subset \PP^{n}$ of codimension $c$ defined by the
maximal minors of a $t\times (t+c-1)$ homogeneous matrix with entries
homogeneous polynomials of degree $a_j-b_i$.

\vskip 2mm  In \cite{KMMNP} and \cite{KM}, we  addressed the
following 3 crucial problems:
\begin{itemize}
\item[(1)] To determine  the dimension of
$W(\underline{b};\underline{a})$  in terms of $a_j$ and $b_i$,
\item[(2)] Is the closure of $W(\underline{b};\underline{a})$  an
irreducible component of $\Hi ^p(\PP^{n})$? and \item[(3)] Is $\Hi
^p(\PP^{n})$ generically smooth along
$W(\underline{b};\underline{a})$?
\end{itemize}

In \cite{KM} we obtained an upper bound for $\dim
W(\underline{b};\underline{a})$ in terms of $a_j$ and $b_i$ which was achieved
in the cases $2\le c\le 5$ and $n-c>0$ (assuming $char (k)=0$ if $c=5$), and
in codimension $c>5$ provided certain numerical conditions are satisfied (See
\cite{KM}, Theorems 3.5 and 4.5; and Corollaries 4.7, 4.10 and 4.14).
Concerning problems (2) and (3), we gave in \cite{KM} an affirmative answer to
both questions in the range $2\le c \le 4$ and $n-c\ge 2$, and in the cases
$c\ge 5$ and $n-c\ge 1$ provided certain numerical assumptions are verified
(See \cite{KM}, Corollaries 5.3, 5.7, 5.9 and 5.10. See also \cite{elli},
\cite{KMMNP} for the cases $2 \le c \le 3$). Note that since every element of
$W(\underline{b};\underline{a})$ has the same Hilbert function, the assumption
$n>c$ is close to being necessary for problem (2). Indeed if $n=c$ the
problems (2) and (3) become more natural provided we replace $\Hi ^p(\PP^{n})$
by the postulation Hilbert scheme, see \cite{K09}.

\vskip 2mm In this work we attempt to extend and complete the results of
\cite{KMMNP} and \cite{KM}. Indeed if $ a_{t+3}> a_{t-2}$ we almost solve
problem (1) in Theorem~\ref{codcdim0} while Theorem~\ref{codcomp} and
Corollary~\ref{cod6}, for $c > 4$, generalize results of \cite{KM} for the
problems (2) and (3) substantially. To prove these results we use induction on
the codimension by successively deleting the columns of the highest degree and
the Eagon-Northcott complex 
associated to a standard determinantal scheme. We also use the theory of Hilbert
flag schemes and the depth of certain mixed determinantal schemes (see Theorem
\ref{dim-mix-det}). We end the paper with two conjectures which are supported
by our results and by a huge number of examples computed using Macaulay 2.

\vskip 2mm  \noindent {\bf Notation:} Throughout this paper
$\PP^n$ is the $n$-dimensional projective space over an
algebraically closed field $k$, $R=k[x_0, x_1, \dots ,x_n]$  and
$\goth m= (x_0, \dots ,x_n)$. By $\cH om_{\cO_X}(\cF,\cG)$ we
denote the sheaf of local morphisms between coherent
$\cO_X$-modules while $\Hom(\cF,\cG)$ denotes the group of
morphisms from $\cF$ to $\cG$. Moreover we set
$\hom(\cF,\cG)=\dim_k\Hom(\cF,\cG)$ and we correspondingly use
small letters for the dimension, as a $k$-vector space, of similar
groups. For any quotient $A$ of $R$ of codimension $c$, we let
$K_A=\Ext^c_R (A,R)(-n-1)$.

In the sequel, $_{\mu }\!\Hom_{R}(M,N)$ denotes homomorphisms of degree $\mu $
of graded $R$-modules. Moreover, we denote the Hilbert scheme by $\Hi
^p(\PP^n)$, $p$ the Hilbert polynomial, and $(X) \in \Hi ^p(\PP^n)$ the point
which corresponds to the subscheme $X\subset \PP^n$ with Hilbert polynomial
$p$. We denote by $I_X$ the saturated homogeneous ideal of $X \subset \PP^n$.
We say that $X$ is {\it general} in some irreducible subset $W \subset \Hi
^p(\PP^n)$ if $(X)$ belongs to a sufficiently small
open subset  $U$ of $W$ (small enough so that any $(X) \in U$ has all the
openness properties that we want to require).


\section{Preliminaries}

This section provides the background and basic results on standard
determinantal ideals, good determinantal ideals and mixed
determinantal ideals needed in the sequel. We refer to \cite{b-v}, \cite{eise}
\cite{KM}  and \cite{M} for the details.

\vskip 2mm Let $\cA=(f_{ij})_{i=1,...t}^{j=0,...,t+c-2}$, $\deg
f_{ij}=a_j-b_{i}$, be a $t\times (t+c-1)$ homogeneous matrix and
let
\begin{equation}\label{gradedmorfismo} \varphi:F=\bigoplus
  _{i=1}^tR(b_i)\longrightarrow G:=\bigoplus_{j=0}^{t+c-2}R(a_j)
\end{equation}
be the graded morphism of free $R$-modules represented by the transpose, $\cA
^{t}$, of $\cA$. Let $I(\cA)=I_t(\cA)$ be the ideal of $R$ generated by the
maximal minors of $\cA$. A codimension $c$ subscheme $X\subset \PP^{n}$ is
said to be {\em standard determinantal} if $I_X=I(\cA)$ for some homogeneous
$t\times (t+c-1)$ matrix $\cA$ as above. Moreover $X$ is \emph{good
  determinantal} if $X$ is standard determinantal and a generic complete
intersection in $ \PP^{n}$ (\cite{KMNP}, Theorem 3.4). In this paper we
suppose $c\ge 2$, $t\ge 2$, $b_1 \le ... \le b_t$ and $ a_0 \le a_1\le ... \le
a_{t+c-2}$. (Note that the case $t=1$ for determinantal schemes corresponds to
the well-known complete
intersections). 

\vskip 2mm Let $W(\underline{b};\underline{a})$ (resp.
$W_s(\underline{b};\underline{a}))$ be the stratum in $ {\rm Hilb}
^p(\PP^{n})$ consisting of good (resp. standard) determinantal schemes as
above. Since our definition does not assume $\cA$ to be minimal (i.e.
$f_{ij}=0$ when $b_{i}=a_{j}$) for $X=\Proj(R/I_{t}(\cA))\in
W(\underline{b};\underline{a})$ (or $W_s(\underline{b};\underline{a})$), we
must reconsider Corollary 2.6 of \cite{KM} where $\cA$ was supposed minimal in
the proof (a slight correction to \cite{KMMNP} and \cite{KM}!). We may,
however, use that proof to see that
\begin{equation} \label{WWs} W(\underline{b};\underline{a}) \ne \emptyset \ \
  \Leftrightarrow \ \ W_s(\underline{b};\underline{a}) \ne \emptyset \ \
  \Leftrightarrow \ \ a_{i-1} \ge b_i \ \ {\rm for \ all} \ i \ {\rm and \ }
    a_{i-1} > b_i \ \ {\rm for \ some } \ i .
\end{equation}

Indeed if we assume the converse of the condition on the right hand side, then
either $I_{t}(\cA) \ni 1$ or one of the maximal minors vanishes, i.e.
$\Proj(R/I_{t}(\cA)) \notin W_s(\underline{b};\underline{a})$. Conversely
assuming the right hand side condition (to simplify notations, assume $a_{i-1}
\ge b_i$ for $1 \le i \le s$ and $a_{i-1} > b_i$ for $s < i \le t$ for some
integer $s < t$), then we may
take $\cA =$ $\left(\begin{smallmatrix}
    I & O \\
    O & \cA'
  \end{smallmatrix} \right)$ where $I$ is the $s \times s$ identity matrix,
$O$ are matrices of zero's and $\cA'$, for $t-s > 1$, the $(t-s) \times
(t-s+c-1)$ matrix used in \cite{KM}, Corollary 2.6 to define a good
determinantal scheme (if $t-s=1$ we take the
entries of $\cA'$ to be a regular sequence). We get $\Proj(R/I_{t}(\cA)) \in
W(\underline{b};\underline{a})$ and we easily deduce \eqref{WWs}. Note that by
\cite{KM}, end of p.\! 2877 and \cite{KM}, Remark 3.7 we still have that the closures  of $W(\underline{b};\underline{a})$ and $
W_s(\underline{b};\underline{a})$ in $ \Hi ^p
(\PP^n)$ 
are equal and irreducible.

\vskip 2mm Let $A=R/I_X$ be the homogeneous coordinate ring of a standard
determinantal scheme. By \cite{b-v}, Theorem 2.20 and \cite{eise},
Corollaries A2.12 and A2.13 the {\em Eagon-Northcott complex} yields a
minimal free resolution of $A$
\begin{equation}\label{EN}0 \ra \wedge^{t+c-1}G^* \otimes S_{c-1}(F)\otimes
  \wedge^tF\ra \wedge^{t+c-2} G ^*\otimes S _{c-2}(F)\otimes \wedge ^tF\ra
  \ldots  \end{equation}
$$ \ra
\wedge^{t}G^* \otimes S_{0}(F)\otimes \wedge^tF\ra R \ra A \ra 0
$$ which allows us to deduce that
any standard determinantal scheme is arithmetically Cohen-Macaulay (ACM).
Moreover if $M_{\cA }:= \coker (\varphi^*)$ then $K_{A}(n+1)\cong
S_{c-1}M_{\cA}(\ell_c)$ where
\begin{equation}\label{ell}
 \ell_i :=\sum_{j=0}^{t+i-2}a_j-\sum_{k=1}^tb_k \ \ {\rm for} \ \ 2 \le i \le c.
\end{equation}

Let ${\cB}$ be the matrix obtained by deleting the last column of ${\cA}$, let
$B = R/I_{B}$ be the $k$-algebra given by the maximal minors of ${\cB}$ and
let $M_{\cB}$ be the cokernel of $\phi^*=\Hom_R(\phi,R)$ where $ \phi:F=\oplus
_{i=1}^tR(b_i)\rightarrow G':=\oplus _{j=0}^{t+c-3}R(a_j)$ is the graded
morphism induced by ${\cB}^t$. Recall that if $c > 2$ there is an exact
sequence
\begin{equation}\label{Mi}
0\longrightarrow B \longrightarrow M_{\cB}(a_{t+c-2})
\longrightarrow M_{\cA}(a_{t+c-2}) \longrightarrow 0
\end{equation}
in which $B \longrightarrow M_{\cB}(a_{t+c-2})$ is a regular section given by
the last column of ${\cA}$. Moreover,
\begin{equation}\label{Di}
0\longrightarrow
M_{\cB}(a_{t+c-2})^* :=\Hom_{B}(M_{\cB}(a_{t+c-2}),B)\longrightarrow B
\longrightarrow A \longrightarrow 0
\end{equation}
is exact by \cite{KMNP} or \cite{KMMNP} (e.g.\! see the text after (3.1) of
\cite{KMMNP}). Note that the proofs of \eqref{Mi}-\eqref{Di} rely
heavily on the equality $\Ann(M_{\cB})=I_B$ established in \cite{BE}. If $c=2$
($\codim_RB=1$) we have at least $\Ann(M_{\cB})I_{t-1}(\cB) \subset I_B
\subset \Ann(M_{\cB})$ by \cite{BE}; thus the kernel, $I_{A/B}$, of $B \to A$
satisfies $I_{A/B}=M_{\cB}(a_{t+c-2})^*$ (resp. $\tilde I_{A/B} \arrowvert_U =
\tilde M_{\cB}(a_{t+c-2})^* \arrowvert_U$ where $U=Y-V(I_{t-1}(\cB))$) for
$c>2$ (resp. $c=2$). Due to the Buchsbaum-Rim resolution of $M:=M_{\cA}$, $M$
is a maximal Cohen-Macaulay $A$-module, and so is $I_{A/B}$ for $c>2$ by
(\ref{Di}).

By successively deleting columns from the right hand side of $\cA$, and
taking maximal minors, one gets a flag of standard determinantal subschemes
\begin{equation}\label{flag}
  ({\mathbf X.}) :  X = X_c \subset X_{c-1} \subset ...  \subset X_{2} \subset
  X_{1} \subset  \PP^{n}
\end{equation}
where each $X_{i+1} \subset X_i$ (with ideal sheaf ${ \mathcal
  I_{X_{i+1}/X_i}} = {\mathcal I_i}$) is of codimension 1, $X_{i} \subset
\PP^{n}$ is of codimension $i$ ($i=1,\dots , c$) and there exist ${\mathcal
  O}_{X_i}$-modules ${\mathcal M_i}$ fitting into short exact
sequences $$0\rightarrow {\mathcal O}_{X_i}(-a_{t+i-1})\rightarrow {\mathcal
  M}_i \rightarrow {\mathcal M}_{i+1} \rightarrow 0 \ \text{ for } \ 2 \leq i
\leq c-1,$$ such that ${\mathcal I}_i (a_{t+i-1})$ is the ${\mathcal
  O}_{X_i}$-dual of ${\mathcal M}_i$ for $2 \leq i \leq c$ (all this holds
also for $i=1$ provided we restrict the sheaves to $X_1-V(I_{t-1}(\varphi_1))$
where $\varphi_1$ is given by $X_1=V(I_{t}(\varphi_1))$). In this context,
we let $D_i:= R/I_{X_i}$, $I_{D_i} := I_{X_i}$ and $I_i:=I_{D_{i+1}}/I_{D_i}$.

\begin{remark} \label{dep} Assume $t \ge 2$, $b_1\le ... \le b_t$, $a_0\le a_1
  \le ... \le a_{t+c-2}$ and let $\alpha \ge 1$ be an integer. If $X$ is
  general in
  $W(\underline{b};\underline{a})$ and $a_{i-\min
    (\alpha ,t)}-b_i\ge 0$ for $\min (\alpha ,t)\le i\le t$, then
\begin{equation}\label{a-b} \codim_{X_j}Sing(X_j)\ge \min\{2\alpha
  -1, j+2\}  \text{ for all } j=2, \cdots , c. \end{equation} This follows 
  from the theorem of \cite{chang}, arguing as in \cite{chang}, Example 2.1. In
particular, if $\alpha \ge 3$, we get for each $i>0$ that
$X_i\hookrightarrow  \PP^{n}$ and $X_{i+1}\hookrightarrow X_{i}$
are local complete intersections (l.c.i.'s) outside some set $Z_i$
of codimension at least $ \min(4,i+1)$ in $X_{i+1}$, cf. the paragraph below.
Finally note that it is easy to show \eqref{a-b} for $(j,\alpha)=(1,2)$
provided  $a_{i-2} > b_i$ for $2 \le i\le t$, cf. \cite{B}, (1.10).
\end{remark}

\vskip 2mm Now let $Z \subset X$ (resp. $Z_i\subset X_i$) be some closed
subset such that $U:=X-Z\hookrightarrow \PP^{n}$ (resp.
$U_{i}:=X_{i}-Z_{i}\hookrightarrow \PP^{n}$) is an l.c.i. By the fact that the
$1^{st}$ Fitting ideal of $M$ is equal to $I_{t-1}(\varphi)$, we get that
$\tilde{M}$ is locally free of rank 1 precisely on $X-V(I_{t-1}(\varphi))$
\cite{BH}, Lemma 1.4.8. Since the set of non-l.c.i points of $X\hookrightarrow
\PP^{n}$ is precisely $V(I_{t-1}(\varphi))$ by e.g. \cite{ulr}, Lemma 1.8, we
get that $U\subset X-V(I_{t-1}(\varphi))$ and that $\tilde{M}$ is locally free
on $U$. Indeed ${\mathcal M}_i$ and $\cI _{X_{i}}/\cI ^2_{X_{i}}$ are locally
free on $U_i$, as well as on $U_{i-1} \cap X_i$. Note also that since
$V(I_{t-1}(\cB)) \subset V(I_{t}(\cA))$, we may suppose $Z_{i-1} \subset X_i$.

Let us  recall the following useful comparison of cohomology
groups. If $L$ and $N$ are finitely generated $A$-modules such
that $\depth_{I(Z)}L\ge r+1$ and $\tilde{N}$ is locally free on
$U:=X-Z$, then the natural map
\begin{equation} \label{NM}
\Ext^{i}_A(N,L)\longrightarrow
H_{*}^{i}(U,{\cH}om_{{\cO}_X}(\tilde{N},\tilde{L}))
\end{equation}
is an isomorphism, (resp. an injection) for $i<r$ (resp. $i=r$)
cf. \cite{SGA2}, exp. VI. Note that we interpret $I(Z)$ as $\goth m$ if $Z=
\emptyset$.

\vskip 2mm In \cite{KM}, Conjecture 6.1,  we conjectured the
dimension of (a non-empty) $ W(\underline{b};\underline{a})$ in
terms of the invariant {\small
\begin{equation} \label{lamda} \lambda_c:= \sum_{i,j}
    \binom{a_i-b_j+n}{n} + \sum_{i,j} \binom{b_j-a_i+n}{n} - \sum _{i,j}
    \binom{a_i-a_j+n}{n}- \sum _{i,j} \binom{b_i-b_j+n}{n} + 1.
\end{equation}}
Here the indices belonging to $a_j$ (resp. $b_i$) range over
$0\le j \le t+c-2$ (resp. $1\le i \le t$) and $\binom{a}{b} =0$
whenever $a<b$.

\begin{conjecture} \label{conj1}
  Given integers $a_0\le a_1\le ... \le a_{t+c-2}$ and $b_1\le ...\le b_t$, we
  set $\ell_i :=\sum_{j=0}^{t+i-2}a_j-\sum_{k=1}^tb_k$ and $h_{i-3}:=
  2a_{t+i-2}-\ell_i +n$, for $i=3,4,...,c$. 
  Assume $a_{i-\min ([c/2]+1,t)}\ge b_{i}$ for $\min ([c/2]+1,t)\le i \le t$.
  Then we have
 \[
 \dim W(\underline{b};\underline{a}) = \lambda_c+ K_3 + K_4+...+K_c \]
where $K_3=\binom{h_0}{n}$ and $K_4= \sum_{j=0}^{t+1} \binom{h_1+a_j}{n}-
\sum_{i=1}^{t} \binom{h_1+b_i}{n}$ and in general \[ K_{i+3}= \sum _{r+s=i
  \atop r , s \ge 0} \sum _{0\le i_1< ...< i_{r}\le t+i \atop 1\le
  j_1\le...\le j_s \le t } (-1)^{i-r} \binom{h_i+a_{i_1}+\cdots
  +a_{i_r}+b_{j_1}+\cdots +b_{j_s} }{n} \  {\rm for}  \ 0 \le i \le c-3. \]
\end{conjecture}

In \cite{KM} we proved that the right hand side in the formula for
$\dim W(\underline{b};\underline{a})$ in Conjecture~\ref{conj1} is
always an upper bound for $\dim W(\underline{b};\underline{a})$
(\cite{KM}, Theorem 3.5) and, moreover,  that Conjecture~\ref{conj1} holds
in the range
\begin{equation} \label{2c5}
2 \le c \le 5 \  \
  {\rm and} \ \ n-c > 0 \ {\rm (\, supposing \ } char (k) = 0 {\rm \ if} \ c =
    5\, )\, ,
  \end{equation}
  cf. \cite{KM}, Theorems 3.5 and 4.5; and Corollaries 4.7, 4.10 and 4.14.
  See also \cite{elli}, \cite{KMMNP} for the cases $2 \le c \le 3$. In
  \cite{K09}, however, the first author gave a counterexample to
  Conjecture~\ref{conj1} for zero-dimensional schemes (see the section of
  conjectures in this paper).

  \begin{remark} \label{dim0new} Conjecture~\ref{conj1} holds in codimension
    $c\ge 6$ provided the numerical conditions of \cite{KM}, Corollary 4.15
    are satisfied. By \cite{KM}, Remarks 4.16, 4.17 and Corollary 4.18
    Conjecture~\ref{conj1} holds (without assuming $n > c$ or $ char (k) = 0$)
    for any $ W(\underline{b};\underline{a})$ satisfying
$$
a_{t+3}>a_{t-1}+a_t+a_{t+1}-a_0-a_1 \qquad \qquad \ \text{ in codimension} \ \
c=5 \ , $$
$$  a_{t+2}>a_{t-1}+a_t-a_0  \qquad \qquad  \qquad  \qquad  \qquad \qquad \
\text{ for }  \ \ n = c=4 \ , $$
 $$  a_{t+1}>a_{t-1} \  \text{ and} \  \ a_{i-2} > b_i \
 \text{ for} \ 2 \le i\le t  \qquad \qquad \text{for} \quad n= c=3 \ . $$ Note
 that in the case $n = c = 3 $ we need $a_{i-2} > b_i$ (not $a_{i-2} \ge b_i$
 as assumed in \cite{KM}) for $2 \le i\le t$ for the proof of \cite{KM},
 Corollary\! 4.18 to hold, see the last sentence of Remark~\ref{dep}.
 \end{remark}

 \vskip 2mm A goal of this paper is to extend the results summarized in
 \eqref{2c5} and in the above remark. To do this we will need the following
 result where $\underline{a} = a_0, a_1,..., a_{t+c-2}$ and $\underline{a'} =
 a_0, a_1,..., a_{t+c-3}$.

\begin{proposition}\label{main1} \ Let $c\ge 3$, let $(X) \in
  W(\underline{b};\underline{a})$ and suppose \ $ \dim
  W(\underline{b};\underline{a'}) \ge \lambda_{c-1}+ K_3 + K_4+...+K_{c-1} $
  and $\depth_{I(Z)}B\ge 2$ for a general $Y=\Proj(B)\in
  W(\underline{b};\underline{a'})$. If
\begin{equation}\label{maineq}
  _0\!
  \hom_R(I_{Y},I_{X/Y}) \le  \sum _{j=0}^{t+c-3} \binom{a_j-a_{t+c-2}+n}{n} \ ,
\end{equation}
then \ \ $\dim W(\underline{b};\underline{a})= \lambda_c+ K_3 +
K_4+...+K_c$ and  \eqref{maineq} turns out to be an equality.
\end{proposition}
\begin{proof} See \cite{K09}, Proposition 3.4.
\end{proof}

Let us recall how we proved Remark~\ref{dim0new} since we want to
generalize that approach. Letting $ a=a_{t+i-2}-a_{t+i-1}$ we
showed in \cite{KM} that $ \Hom_{D_{i-1}}(I_{i-1}, I_{i}) \cong
D_i(a)$ for $i < c$ using $\depth_{I(Z_{i-1})}D_{i}\ge 2$.
Indeed this follows from \eqref{NM}, i.e. from
\begin{equation*} H^0_{*}(U_{i-1},\cH om(\cI_{i-1},\cI_i))\cong
  H^0_{*}(U_{i-1},\cH om_{\cO _{X_{i}}}(\cI _{i-1}\otimes_{\cO _{X_{i-1}}}\cO_
  {X_i}\otimes \cI ^* _i, \cO _{X_{i}}))\end{equation*} because $
\tilde{M}_i|_{U_{i-1}} \cong \cI ^*_i(-a_{t+i-1})|_{U_{i-1}}$ is locally free,
$\tilde{M}_{i} |_{U_{i-1}} \cong \tilde{M}_{i-1} \otimes \cO_ {X_i}
|_{U_{i-1}}$ and hence $\cI _{i-1}\otimes_{\cO _{X_{i-1}}}\cO_{X_{i}}
|_{U_{i-1}}\cong \cI_{i}(-a)|_{U_{i-1}}$, and note that
$U_{i-1} \cap X_i \subset U_{i}$. Then since
\begin{equation}\label{1}
  0\rightarrow \Hom _R(I_{i-1}, I_{i}) \rightarrow \Hom
  _R(I_{D_{i}}, I_{i}) \rightarrow \Hom _R(I_{D_{i-1}},
  I_{i})\end{equation} is exact  we were able to show \eqref{maineq} for
$X=X_{c} \subset Y=X_{c-1}$ by putting $\ _0\! \Hom _R(I_{D_{c-2}}, I_{c-1})=0$
(through making the minimal generators of $I_{D_{c-2}}$ and  $I_{c-1}$
explicit via \eqref{EN}).
Using Proposition~\ref{main1} we obtained Remark~\ref{dim0new}
if
the conjecture holds for $
W(\underline{b};\underline{a'}) \ni Y:=X_{c-1}$.

\vskip 2mm Let us finish this section by gathering all the results on the depth
of mixed determinantal ideals and cogenerated ideals needed in the next section.
We start by fixing some notation. Let
\begin{equation}\label{defmatriz} \cA= \left ( \begin{array}{ccccccc} f_1^1 & f_2^1 & \cdots & f_q^1 \\ f_1^2 & f_2^2 & \cdots & f_q^2 \\ \vdots & \vdots & \vdots & \vdots \\ f_1^p & f_2^p & \cdots & f_q^p \\
    \end{array} \right )  \end{equation}
be a homogeneous $p\times q$ matrix with entries homogeneous polynomials $f_i^j\in k[x_0, \cdots ,x_n]$ of degree $a_j-b_i$. For any choice $(\underline{\alpha} ; \underline{\beta} )=( \alpha _1, \cdots , \alpha _m ;  \beta _1, \cdots , \beta _m)$ of row indexes $1\le \alpha _1< \cdots < \alpha _m\le p$ and of column indexes
$1\le  \beta _1< \cdots < \beta _m\le q$,  we denote 
by $I_{(\underline{\alpha };\underline{ \beta} )}(\cA )$
the {\em ideal cogenerated by } $(\underline{\alpha} ;\underline{ \beta} )$,
i.e. $I_{(\underline{\alpha} ,\underline{ \beta} )}$ is the homogeneous ideal
of $k[x_0, \cdots ,x_n]$ generated by all $(m+1)\times (m+1)$ minors of $\cA$,
all $i \times i$ minors of the rows $1, \cdots , \alpha _i -1$ for $i=1, \cdots ,
m$ and all $i \times i$ minors of the columns $1, \cdots, \beta _i -1$ for $i=1,
\cdots , m$.

\begin{example} (1) If $(\underline{\alpha };\underline{ \beta} )=(1, \cdots ,
  t-1;1,\cdots , t-1)$, then $I_{(\underline{\alpha }; \underline{\beta
    })}(\cA )$ is the ideal generated by the $t \times t$ minors of $\cA$,
  i.e., $I_{(\underline{\alpha} ; \underline{ \beta })}(\cA )=I_t(\cA)$.

  (2) Let $\cA$ be a homogeneous $p\times q$ matrix and let $\cA_i$ be the
  matrix obtained by deleting the last $i-1$ columns of $\cA$. We assume
  $p\le q$ and we fix an integer $m<p$ and $(\underline{\alpha} ;\underline{
    \beta} )=(1, \cdots , m; \beta _1,\cdots , \beta _m)$. Set $j:=\min\{i \mid
  \beta _i>i \}$, $c_1=1$ and $c_s= q+2-\beta _{m+2-s}$ for $2\le s\le m+2-j$.
  The ideal cogenerated by $(\underline{\alpha} ;\underline{ \beta} )=(1,
  \cdots , m; \beta _1,\cdots , \beta _m)$ can be identified with the
  following mixed determinantal ideal
$$I_{(\underline{\alpha};\underline{\beta})}(\cA)=I_{m+1}(\cA _{c_1})+
I_{m}(\cA _{c_2})+...+I_{j}(\cA _{c_{m+2-j}}),
$$where $I_{\lambda}(\cA _{\varrho})$ denotes the ideal
generated by all $\lambda \times \lambda$ minors of $\cA _{\varrho }$.
\end{example}

\begin{theorem}\label{mixedindeterminates} Let $\cA =(x_{i,j})$ be a $p\times q$ matrix of indeterminates and let  $(\underline{\alpha };\underline{ \beta} )=( \alpha _1, \cdots , \alpha _m ;  \beta _1, \cdots , \beta _m)$ be a choice of $m$ row indexes  and of $m$ column indexes. Let $I_{(\underline{\alpha };\underline{ \beta } )}$ be the ideal cogenerated by $(\underline{\alpha };\underline{ \beta } )$. Then, $I_{(\underline{\alpha };\underline{ \beta} )}(\cA )$  is a Cohen-Macaulay ideal and
\begin{equation}\label{ht} ht (I_{(\underline{\alpha };\underline{ \beta })}(\cA))= pq-(p+q+1)m+\sum _{i=1}^m (\alpha _i+\beta _i).\end{equation}
\end{theorem}
\begin{proof} See \cite{b-v}, Corollary 5.12.
\end{proof}

We are now ready to state the main result of this section and to prove that
the above formula for the height of a cogenerated ideal associated to a matrix
with entries that are indeterminates also works for a general homogeneous
matrix with entries thar are homogeneous polynomials of positive degree.
Indeed, we have

\begin{theorem} \label{dim-mix-det} Fix integers $b_1, \cdots , b_q$ and $a_1,
  \cdots ,a_p$. Let $\cA =(f_i^{j})_{i=1, \cdots, q}^{j=1,\cdots , p}$ be a
  general homogeneous $p\times q$ matrix $\cA $ with entries that are
  homogeneous forms of degree $a_j-b_i$, and assume that $a_j>b_i$ for all
  $j,i$. Let $(\underline{\alpha };\underline{ \beta })=( \alpha _1, \cdots ,
  \alpha _m ; \beta _1, \cdots , \beta _m)$ be any choice of $m$ row indexes
  and of $m$ column indexes, and suppose $n+1\ge pq-(p+q+1)m+\sum _{i=1}^m
  (\alpha _i+\beta _i)$. Then $I_{(\underline{\alpha };\underline{ \beta
    })}(\cA )$ is a Cohen-Macaulay ideal and $$ht (I_{(\underline{\alpha
    };\underline{ \beta })}(\cA))= pq-(p+q+1)m+\sum _{i=1}^m (\alpha _i+\beta
  _i).$$
\end{theorem}
\begin{proof} We clearly have $ht (I_{(\underline{\alpha} ;\underline{ \beta
    })}(\cA))\le  pq-(p+q+1)m+\sum _{i=1}^m (\alpha _i+\beta _i)$ and it will
  be enough to construct an example of homogeneous $p\times q$ matrix $\cM $
  with entries homogeneous forms $f_i^j\in k[x_0,\cdots ,x_n]$ such that  the
  ideal $I_{(\underline{\alpha} ;\underline{ \beta })}(\cM)$ cogenerated by
  $(\underline{\alpha };\underline{ \beta })$ is Cohen-Macaulay and  $ht
  (I_{(\underline{\alpha} ;\underline{ \beta })}(\cM) )= pq-(p+q+1)m+\sum
  _{i=1}^m (\alpha _i+\beta _i).$ Therefore, let $n+1\ge pq-(p+q+1)m+\sum
  _{i=1}^m (\alpha _i+\beta _i).$ We will distinguish 2 cases: 

\vskip 2mm
\noindent \underline{Case 1.} Assume $a_j-b_i=1$ for all $i,j$ (i.e. the
entries of the matrix $\cA$ are linear forms). By Theorem
\ref{mixedindeterminates} for any choice of $p$, $q$ and $(\underline{\alpha
};\underline{ \beta })$, we have the matrix $\overline{ \cA }=(x_{i,j})$ of
indeterminates and the ideal $I _{(\underline{\alpha };\underline{ \beta}
  )}(\overline{ \cA })\subset S:=k[x_{1,1}\cdots , x_{p,q}]$ cogenerated by
$(\underline{\alpha} ;\underline{ \beta} )$ is Cohen-Macaulay of height $$ht
(I_{(\underline{\alpha };\underline{ \beta })}(\overline{ \cA }))=
pq-(p+q+1)m+\sum _{i=1}^m (\alpha _i+\beta _i).$$ We choose $pq-n-1$  general
linear forms $\ell _1, \cdots , \ell_{pq-n-1}\in S=k[x_{1,1}\cdots , x_{p,q}]$
and we set $S/(\ell _1,\cdots ,\ell _{pq-n-1})\cong k[x_0,x_1, \cdots
,x_n]=:R$. Let us call $I\subset R$ the ideal of $R$ isomorphic to the ideal
$I_{(\underline{\alpha };\underline{ \beta })}(\overline{ \cA })/(\ell
_1,\cdots ,\ell _{pq-n-1})$ of $S/(\ell _1,\cdots ,\ell _{pq-n-1})$. Note that
$I$ is a Cohen-Macaulay ideal of height $ht(I)= pq-(p+q+1)m+\sum _{i=1}^m
(\alpha _i+\beta _i)$ and, in addition, $I$ is nothing but the ideal
$I_{(\underline{\alpha };\underline{ \beta} )}(\cM)$ where $\cM =(m_i^j)$ is a
$p\times q$ homogeneous matrix with entries linear forms in $k[x_0, x_1,
\cdots , x_n]$ obtained from $\overline{ \cA }=(x_{i,j})$ by substituting
using the equations $\ell _1,\cdots ,\ell _{pq-n-1}$, which proves what we
want. 

\vskip 2mm
\noindent \underline{Case 2.} Assume $a_j-b_i\ge 1$ for all $i,j$. In this
case, it is enough to raise the entry $m_i^j$ of the above matrix $\cM$ to the
power $a_j-b_i$.
\end{proof}

\begin{remark} \label{keyrmk} Fix  integers $a_0\le a_1\le ... \le a_{t+c-2}$ and $b_1\le ...\le b_t$. Let $X  \subset \PP^n$, $(X) \in
  W(\underline{b};\underline{a})$ be a general determinantal ideal associated to a $t \times (t+c-1)$ matrix $\cA$ represented by a
  graded morphism as in (\ref{gradedmorfismo}).
  Let
$$
  ({\mathbf X.}) :  X = X_c \subset X_{c-1} \subset ...  \subset X_{2} \subset
  X_{1} \subset  \PP^{n}
$$ be the flag of standard determinantal subschemes that we obtain by
successively deleting columns from the right hand side of $\cA$
and let $\varphi _i$ be the graded morphism associated to the matrix
which defines $X_i$. Assume $a_0>b_t$ and $c\ge 3$. Applying Theorem
\ref{dim-mix-det}, we get the following formula
\begin{equation} \label{assump}
\dim R/(I_{t-1}(\varphi_1) + I_{t}(\varphi_{c-1})) = \dim
  D_{c-1} - 2 \ .
\end{equation}
Even more, if $ \dim D_{c-1} \ge 3$ and $c\ge 4$ we have the equalities
\begin{equation} \label{assump2}
\dim R/(I_{t-1}(\varphi_i) + I_{t}(\varphi_{c-1})) = \dim
  D_{c-1} - i-1
\end{equation}
for $1 \le i \le 2$ which will play an important role in the next section.
\end{remark}

\begin{remark} \label{dim00new} If $t=2$ we can prove \eqref{assump} directly
  as follows. Let $\cA =[\cC,\cD,v]$ where $ \cC$ is a 2 by 2 general
  enough matrix in the variables $x_0,x_1,x_2,x_3$ (e.g. with rows
  $(x_0^{a_0-b_1}, x_1^{a_1-b_1})$ and $(x_2^{a_0-b_2}, x_3^{a_1-b_2})$ ), $v$
  is some column and $\cD$ is a 2 by $c-2$ matrix whose first row is
  $(x_4^{a_2-b_1}, x_5^{a_3-b_1}, ...,x_c^{a_{c-2}-b_1},0)$ and whose second
  row is $(0,x_4^{a_3-b_2}, x_5^{a_4-b_2},...,x_c^{a_{c-1}-b_2})$. Note that
  $\varphi_1$ corresponds to $\cC$ and $\varphi_{c-1}$ to $[\cC,\cD]$. Since
  $\cC$ is general, we get $\codim_R R/I_{t-1}(\varphi_1)=4$, i.e. the radical
  of $I_{t-1}(\varphi_1)$ is $(x_0,x_1,x_2,x_3)$. Since it is clear that the
  radical of the ideal generated by the maximal minors of $\cD$ is
  $(x_4,x_5,...,x_c)$, it follows that the scheme $X_{c-1}= \Proj(D_{c-1}) \in
  W(\underline{b};\underline{a'})$ given by the maximal minors of $[\cC,\cD]$
  satisfies \eqref{assump}. The general element $X$ of $
  W(\underline{b};\underline{a})$ will therefore have a flag where $X_{c-1}$
  satisfies \eqref{assump}.
\end{remark}


\section{Smoothness and  dimension of the determinantal locus}

This section is the heart of the paper and contains the results which
generalize quite a lot of our previous contributions to problems (1)-(3)
stated in the introduction.

\begin{proposition}\label{mainnewprop} With notation as in
  Remark~\ref{keyrmk}, let $c \ge 3$ and suppose \eqref{assump} (this holds if
  $a_{0} > b_t$). If $ a_{t+c-2}> a_{t-2} \ $ then \eqref{maineq} holds for
  $X:=X_{c} \subset Y:=X_{c-1}$, i.e.
\begin{equation*}\label{maineqq}
  _0\!
  \hom_R(I_{D_{c-1}},I_{D_c/D_{c-1}}) \le  \sum _{j=0}^{t+c-3}
  \binom{a_j-a_{t+c-2}+n}{n} .
\end{equation*} In particular we get
$\dim W(\underline{b};\underline{a})= \lambda_c+ K_3 +
K_4+...+K_c$ provided $\dim
W(\underline{b};\underline{a'}) =  \lambda_{c-1}+ K_3 +
K_4+...+K_{c-1}$ where $\underline{a'} = a_0, a_1,...,
a_{t+c-3}$.
\end{proposition}

\begin{proof}
  We {\it claim} that $\depth_{I(Z_{j})}D_{c-1}\ge 2$ for all $j$ satisfying
  $0 < j < c-1$. Indeed by the discussion right after Remark~\ref{dep} we see
  that we may take $I(Z_j)$ to be the ideal $I_{t-1}(\varphi_j) \subset R$ of
  submaximal minors of the matrix which defines $X_j$, and $I(Z_j)D_{c-1}$ to
  be $(I_{t-1}(\varphi_j) + I_{t}(\varphi_{c-1}))D_{c-1}$. It follows that
  $\depth_{I_{t-1}(\varphi_j)}D_{c-1}\ge 2$ for $j=1$, i.e. $\dim D_{c-1} -
  \dim D_{c-1}/(I_{t-1}(\varphi_1) + I_{t}(\varphi_{c-1}))D_{c-1} \ge 2$,
  implies the claim. Hence we conclude the proof of the claim by \eqref{assump}.

  For every $j$, $0 < j < c-1$, put $a:=a_{t+j-1}-a_{t+c-2}$. We {\it claim}
  that $$\Hom_{D_j}(I_{j}, I_{c-1}) \cong D_{c-1}(a) \ . $$ To prove this
  claim we remark that $\depth_{I(Z_{j})}I_{c-1}\ge 2$ since $I_{c-1}$ is
  maximally CM. Using that $\cI_{j}$ is locally free on $U_{j}$ and the
  arguments in the text before \eqref{1}, see the text accompanying \eqref{flag}
  for $j = 1$, we get $$\cI_{j}\otimes_{\cO _{X_{j}}}\cO_{X_{c-1}} |_{U_{j}}
  \cong \cI_{j}\otimes_{\cO _{X_{j}}}\cO_{X_{j+1}}\otimes
  ...\otimes_{\cO_{X_{c-2}}}\cO_{X_{c-1}} |_{U_{j}} \cong
  \cI_{c-1}(-a)|_{U_{j}}.$$ It follows that $\ \cH om_{\cO_{X_{j}}}(\cI _{j},
  \cI_{c-1})(-a) \cong \cH om_{\cO_{X_{c-1}}}(\cI_{c-1}, \cI_{c-1}) \cong
  \cO_{X_{c-1}}$ are isomorphic as sheaves on $U_{j} \cap X_{c-1}$, i.e. we
  get that $H^0_{*}(U_{j},\cH om(\cI _{j}, \cI_{c-1})) \cong
  H^0_{*}(U_{j},\cO_{X_{c-1}})(a)$ and hence the claim from
  $\depth_{I(Z_{j})}D_{c-1}= \depth_{I(Z_{j})}I_{c-1}\ge 2$ and \eqref{NM}.

  Now we repeatedly use the exact sequence
  \begin{equation} \label{47} 0 \rightarrow D_{c-1}(a) \cong \Hom_{D_j}(I_{j},
    I_{c-1}) \rightarrow \Hom_R(I_{D_{j+1}},I_{c-1}) \rightarrow \Hom
    _R(I_{D_{j}},I_{c-1}) \rightarrow \
\end{equation}
for $j=c-2, c-3,...,1$. Since $a_{t+j-1} \le a_{t+c-2}$, we have $\dim
D_{c-1}(a)_0 = \binom{a+n}{n}$. It follows that $$ \ _0\!
\hom(I_{D_{c-1}},I_{c-1}) \le\ _0\! \hom(I_{D_{1}},I_{c-1}) + \sum
_{i=t}^{t+c-3} \binom{a_i-a_{t+c-2}+n}{n} $$ where we have replaced
$a_{t+j-1}$ by $a_i$ in which case $1 \le j \le c-2$ corresponds to $t \le i
\le t+c-3$.

It remains to prove that $ \ _0\! \hom(I_{D_{1}},I_{c-1}) \le \sum _{i=0}^{t-1}
\binom{a_i-a_{t+c-2}+n}{n} $ since we then by Proposition~\ref{main1} get the
dimension formula. Using that $X_1=\Proj(D_1)$ is a hypersurface of degree
$\ell_1$, we find $\ _0\! \hom(I_{D_{1}},I_{c-1}) \cong \ \dim
I_{c-1}(\ell_1)_0$ where $\ell_k=\sum_{j=0}^{t+k-2}a_j-\sum_{i=1}^tb_i$. Now
we have to make the degrees of the minimal generators of $I_{c-1}$
explicit. 
Taking a close look at \eqref{EN}, we see that a minimal generator $f$ of
$I_{c-1} \cong I_{D_c}/I_{D_{c-1}}$ of the smallest possible degree has degree
$s(I_{c-1}):=\ell _c-\sum _{j=t-1}^{t+c-3} a_{j}$ because $a_0\le a_1\le
\cdots \le a_{t+c-2}$. Since $\ell_1-s(I_{c-1})= a_{t-1}-a_{t+c-2} \le 0$ by
the definition of $\ell_k$, we get either $ \dim I_{c-1}(\ell_1)_0=0$ or $
a_{t-1}=a_{t+c-2}$. In the latter case the assumption $ a_{t+c-2}> a_{t-2}$
implies that the degrees of all minimal generators, except for $f$, 
are strictly greater than $s(I_{c-1})$, i.e.
we get $ \dim I_{c-1}(\ell_1)_0=\binom{a_{t-1}-a_{t+c-2}+n}{n} $ and we are
done.
\end{proof}

By repeatedly using Proposition~\ref{mainnewprop}  we get the
\begin{theorem} \label{codcdim0} Let $X \subset \PP^{n}$, $(X) \in
  W(\underline{b};\underline{a})$, be a general determinantal scheme and
  suppose $a_0 > b_t$. Moreover if $c \ge 6$ we suppose $a_{t+3} > a_{t-2}$
  (or $a_{t+4} > a_{t-2}$ provided $char k = 0$) and if $3 \le c \le 5$ we
  suppose $ \ a_{t+c-2}>a_{t-2} $. Then 
  we have $$ \dim W(\underline{b};\underline{a})= \lambda _c+K_3+K_4+ \cdots
  +K_c \ .$$
 \end{theorem}

\begin{proof}
  If $3 \le c \le 6$ ($char k = 0$ if $c=6$) and $ \ a_{t+c-2}>a_{t-2} $ we
  use \eqref{2c5} to find $ \dim W(\underline{b};\underline{a'})$ and we
  conclude the proof by Proposition~\ref{mainnewprop}.

  If $c \ge 6$ and $a_{t+3} > a_{t-2}$ (resp. $a_{t+4} > a_{t-2}$ if $c \ge 7$)
  we repeatedly use Proposition~\ref{mainnewprop} to reduce to the case $c=5$
  (resp. $c=6$) and we conclude by the first part of the proof (note that the
  assumption $ \ a_{t+c-2}>a_{t-2} $ of Proposition~\ref{mainnewprop} is
  satisfied in this induction).
\end{proof}

\begin{remark} \label{dim00new} We expect that the assumption $a_0>b_t$ can be
  weakened in Theorem~\ref{codcdim0}, as well as in \eqref{assump}. At least
  it does for $c=3$ provided we assume $\ a_{i-2} > b_i$ for $ 2 \le i\le t$.
  Indeed since $ I_{t}(\varphi_2) \subset I_{t-1}(\varphi_1)$ we first show
  \eqref{assump} 
  using Remark~\ref{dep}. Then the proof above
  applies to conclude as in Theorem~\ref{codcdim0} provided $ \
  a_{t+1}>a_{t-2} $, cf. Remark~\ref{dim0new}.
\end{remark}

If the condition \eqref{assump2} is satisfied,
 then we can prove the following result for $ \overline
 {W(\underline{b};\underline{a})}$ to be a generically smooth irreducible
 component.
\begin{theorem} \label{codcomp} Let $X \subset \PP^{n}$, $(X) \in
  W(\underline{b};\underline{a})$, be a general determinantal scheme of
  dimension $n-c\ge 1$, let $c > 2$ and let $X=X_c\subset X_{c-1}\subset
  ...\subset X_2\subset \PP^{n}$, $X_i = \Proj(D_i)$, be the flag obtained by
  successively deleting columns from the right hand side. If $a_0 > b_t$,
 $$ _0\! \Ext ^1_{D_2}(I_{D_2}/I^2_{D_2},I_{2})=0 \ \ {and} \ \ _0\! \Ext
 ^1_{D_3}(I_{D_3}/I^2_{D_3},I_{i})=0 \mbox{ for } i=3,...,c-1 \ , $$ then
 $\overline{ W(\underline{b};\underline{a})}$ is a generically smooth
 irreducible component of the Hilbert scheme $\Hi ^p(\PP^{n})$.
\end{theorem}
\begin{remark} \label{3c4} If $n-c \ge 2$ then we have $_0\! \Ext
  ^1_{D_2}(I_{D_2}/I^2_{D_2},I_{2})=0$ provided $c=3$, and $_0\! \Ext
  ^1_{D_2}(I_{D_2}/I^2_{D_2},I_{2}) = \ _0\! \Ext
  ^1_{D_3}(I_{D_3}/I^2_{D_3},I_{3})=0$ provided $c=4$ by \cite{KM},
  (5.4)-(5.8). 
Hence the conclusion of
  Theorem~\ref{codcomp} holds provided $n-c \ge 2$ and $3 \le c \le 4$.
  Moreover if $n-c \ge 1$ and $c>4$ then both $ \Ext$-groups above still
  vanish by \cite{KM} and we may in this case replace the assumption of
  Theorem~\ref{codcomp} given by the displayed formula with $$ \ _0\! \Ext
  ^1_{D_3}(I_{D_3}/I^2_{D_3},I_{i})=0 \mbox{ for } i=4,...,c-1 \ . $$ Note
  that in the case $n-c \ge 1$ and $c=2$, the conclusion of
  Theorem~\ref{codcomp} holds and moreover, $\Hi ^p(\PP^{n})$ is smooth at any
  $(X)\in W(\underline{b};\underline{a})$ by \cite{elli}.
\end{remark}

\begin{remark} For $c > 2$ one knows that
$\Hi ^p(\PP^n)$ is not always smooth at any $(X) \in
W(\underline{b};\underline{ a})$ \cite{MDPi}. Indeed, since
  $W(\underline{b};\underline{a})$ is irreducible, it is not difficult to find
  singular points of $\Hi ^p(\PP^{n})$ by first computing its tangent space
  dimension, $h^0(\cN _X)$, at a
  general $(X) \in W(\underline{b};\underline{a})$, using Macaulay 2. Then by
  experimenting with special choices of $(X_0) \in
  W(\underline{b};\underline{a})$ one may find $h^0(\cN _X)<h^0(\cN _{X_0})$
  which means that $\Hi ^p(\PP^{n})$ is singular at $(X_0)$, see \cite{siq}
  which even computes the obstructions of deformations, using Singular, in a
  related case.
\end{remark}

\begin{proof} Due to Theorem 5.1 of \cite{KM} we must show that 
\begin{equation}\label{exteq}  _0\! \Ext
^1_{D_i}(I_{D_i}/I^2_{D_i},I_{i})=0 \mbox{ for } i=2,...,c-1.
\end{equation} By assumption we need to prove the vanishing \eqref{exteq} for
$i=4,...,c-1$ and $c > 4$. By induction  on $c$ it suffices to show
it for
$i=c-1$, $c \ge 5$. Hence it suffices to see that there exist injections
\begin{equation}\label{exteq2}  _0\!
\Ext^1_{D_{j+1}}(I_{D_{j+1}}/I^2_{D_{j+1}},I_{c-1})\hookrightarrow \ _0\!
\Ext^1_{D_{j}}(I_{D_{j}}/I^2_{D_{j}},I_{c-1}) \mbox{ for } j=2,...,c-2.
\end{equation}
By \eqref{assump2} and the arguments in the first paragraph of
the proof of Proposition~\ref{mainnewprop} we may suppose
$\depth_{I(Z_{j})}D_{c-1}\ge 3$ for all $j$ satisfying $1 < j < c-1$.

We {\it claim} that the left-exact sequence \eqref{47} is also right-exact,
i.e. that the rightmost map of the $\Hom$-groups is surjective for $1 < j <
c-1$. To show this, it certainly suffices to prove $ \
\Ext^1_{R}(I_{j},I_{c-1})=0$. However, by paying closer attention to the
modules of \eqref{47} we shall see that also $ \
\Ext^1_{D_{j}}(I_{j},I_{c-1})=0$ for $1 < j < c-1$ suffices for proving the
claim. Indeed if we apply $ (-) \otimes_R D_j$ to $0 \to I_{D_{j}} \to
I_{D_{j+1}} \to I_{j} \to 0$ we get the right-exact sequence $ I_{D_{j}}
\otimes_R D_j \to I_{D_{j+1}} \otimes_R D_j \to I_{j} \to 0$ where $
I_{D_{j+1}} \otimes_R D_j \cong I_{D_{j+1}}/ I_{D_{j}} \cdot I_{D_{j+1}}$ and
since $I_{D_{j}} \otimes_R D_{j+1} \cong I_{D_{j}}/ I_{D_{j}} \cdot
I_{D_{j+1}}= \ker ( I_{D_{j+1}}/ I_{D_{j}} \cdot I_{D_{j+1}} \to I_{j})$ we
obtain the exact sequence $$ 0 \ra I_{D_{j}} \otimes_R D_{j+1} \ra I_{D_{j+1}}
\otimes_R D_j \ra I_{j} \ra 0$$ to which we apply $ \
\Hom_{D_{j}}(-,I_{c-1})$. Then we get exactly \eqref{47} continued to the
right by $ \ \Ext^1_{D_{j}}(I_{j},I_{c-1})$. Thus the vanishing of $ \
\Ext^1_{D_{j}}(I_{j},I_{c-1})$ implies that \eqref{47} is right-exact.

To see that $ \ \Ext^1_{D_{j}}(I_{j},I_{c-1})=0$ we use the isomorphism $\ \cH
om_{\cO_{X_{j}}}(\cI _{j}, \cI_{c-1})(-a)\arrowvert_{U_{j} \cap X_{c-1}} \cong
\cO_{X_{c-1}}\arrowvert_{U_{j} \cap X_{c-1}}$ which we obtained in the proof of
Proposition~\ref{mainnewprop}. By  (\ref{NM}) it follows that $$ \
\Ext^1_{D_{j}}(I_{j},I_{c-1}) \cong H^1_{*}(U_{j},\cH om(\cI _{j}, \cI_{c-1}))
\cong H^1_{*}(U_{j},\cO_{X_{c-1}}(a)) = 0$$ because we have
$\depth_{I(Z_{j})}D_{c-1}\ge 3$ and hence $\depth_{I(Z_{j})}I_{c-1}\ge 3$
($I_{c-1}$ is maximally CM). This proves the claim.

Now we can rewrite the exact sequence  \eqref{47} as
\begin{equation} \label{57} 0 \rightarrow D_{c-1}(a) \rightarrow
  \Hom_{D_{j+1}}(I_{D_{j+1}}/I_{D_{j+1}}^2,I_{c-1}) \rightarrow \Hom
  _{D_{j}}(I_{D_{j}}/I_{D_{j}}^2,I_{c-1}) \rightarrow \ 0 \ .
\end{equation}
Sheafifying, restricting to ${U_{j} \cap X_{c-1}}$ (note that $\cI_{D_{j+1}}$
is also locally free on ${U_{j} \cap X_{c-1}}$) and taking cohomology, we get

{\small
\begin{equation*}\label{47new}
  \to H^1(U_{j},\cO_{X_{c-1}}(a)) \rightarrow  H^1(U_{j},\cH
  om(\cI_{X_{j+1}}/\cI ^2_{X_{j+1}},\cI_{c-1}))  \rightarrow  H^1(U_{j},\cH
  om(\cI_{X_{j}}/\cI ^2_{X_{j}},\cI_{c-1}))   \rightarrow
\end{equation*}}

\noindent Since $\depth_{I(Z_{j})}D_{c-1}\ge 3$, the two latter
$H^{1}$-groups are by (\ref{NM}) isomorphic to 
the $\ _0\! \Ext ^{1}$-groups quoted in \eqref{exteq2} and since
$H^1(U_{j},\cO_{X_{c-1}}(a))=0$ we are done.
\end{proof}
\begin{remark} \label{remassump2} Since \eqref{assump2} holds also for $i=3$
  provided $\dim D_{c-1} \ge 4$, $c \ge 5$ and $a_0 > b_t$ by Theorem
  \ref{dim-mix-det}, we may continue the proof above to see that the
  injections \eqref{exteq2} are isomorphisms for $j \ge 3$. Hence if $X$ is
  general and $n-c \ge 2$, then$$ _0\!
  \Ext^1_{D_{3}}(I_{D_{3}}/I^2_{D_{3}},I_{c-1}) \cong \ _0\!
  \Ext^1_{D_{c-1}}(I_{D_{c-1}}/I^2_{D_{c-1}},I_{c-1}) \ . $$ Here the leftmost
  $\Ext ^{1}$-group is computed much faster by Macaulay 2 than the rightmost
  one. 
  We also get an injection in \eqref{exteq2} for $j = 2$, but now it
  is not necessarily an isomorphism.
\end{remark}
\begin{corollary} \label{cod6} 
  Let $n -c \ge 1$, $c \ge 5$ and suppose $a_{0} > b_{t}$ and $ \quad
  a_{t+3}>a_{t-1}+a_t-b_1$. Then $ \overline {W(\underline{b};\underline{a})}$
  is a generically smooth irreducible component of $ \ \Hi ^p(\PP^{n})$ of
  dimension $\lambda_c + K_3+...+K_c$.
 \end{corollary}
 \begin{proof} We get $ \dim \overline {W(\underline{b};\underline{a})}$ from
   Theorem~\ref{codcdim0}. Hence by \eqref{exteq2} and Remark~\ref{3c4} it
   suffices to show that $\ _0\! \Ext^1_{R}(I_{D_{2}},I_{i})=0$ for $4 \le i
   \le c-1$ since $\ _0\! \Ext^1_{D_{2}}(I_{D_{2}}/I^2_{D_{2}},I_{i})$ is a
   subgroup of $\ _0\! \Ext^1_{R}(I_{D_{2}},I_{i})$. First let $i=c-1$. By the
   Eagon-Northcott resolution \eqref{EN} we see that the largest possible
   degree of a relation for $I_{D_{2}}$ is $\ell_2-b_1$ and the smallest
   possible degree of a generator of $I_{c-1}\cong I_{D_c}/I_{D_{c-1}}$ is
   $\ell _c-\sum _{j=t-1}^{t+c-3} a_{j}$. Since $\ell _c= \ell _2 +
   \sum_{j=t+1}^{t+c-2}a_j$, we get $\ _0\! \Ext^1_{R}(I_{D_{2}},I_{c-1})=0$
   from $$\ell_2-b_1 < \ell _2 + \sum_{j=t+1}^{t+c-2}a_j-\sum _{j=t-1}^{t+c-3}
   a_{j} = \ell _2 - a_{t-1} -a_t + a_{t+c-2} \ , $$ i.e. from $a_{t+c-2}
   >a_{t-1}+a_t-b_1.$ Since we need the vanishing of $\ _0\!
   \Ext^1_{R}(I_{D_{2}},I_{i})$ for any $i=4,5,...,c-1$, we must suppose
   $a_{t+3} >a_{t-1}+a_t-b_1$ and hence we get the corollary.
\end{proof}

\begin{remark} Note that if $c=3$ (resp. $c=4$) we can argue as above to see
  that the conclusions of Corollary~\ref{cod6} hold provided $a_{t+1}
  >a_{t-1}+a_t-b_1$ (resp. $a_{t+2} >a_{t-1}+a_t-b_1$). This is, however,
  proved in \cite{KM}, Corollary 5.10. For $c \ge 5$, Corollary~\ref{cod6}
  generalizes the corresponding result \cite{KM}, Corollary 5.9 quite a lot.
\end{remark}


\section{Conjectures}

In \cite{K09} the first author discovered a counterexample to
Conjecture~\ref{conj1} for every $c$ in the range $n=c\ge 3$. Indeed the
vanishing all $2 \times 2$ minors of a general $2 \times (c+1)$ matrix of
linear entries defines a reduced scheme  of $c+1$ different points in
$\PP^{c}$. The conjectured dimension of $ W({0,0};{1,1,...,1})$ is
$c(c+1)+c-2$ while its actual dimension is at most $ c(c+1)$.

On the other hand Theorem~\ref{codcdim0} is quite close to proving
Conjecture~\ref{conj1}. The crucial assumption in Theorem~\ref{codcdim0} is
the inequality $ \ a_{t+c-2}>a_{t-2} $ (or $a_{t+3} > a_{t-2}$ if $c > 5$).
Since we, in addition to proving Theorem~\ref{codcdim0}, have computed quite a
lot of examples where we have $ \ a_{t+c-2}=a_{t-2} $ and $a_{i-\min
  ([c/2]+1,t)} > b_{i}$, and each time, except for the counterexample, obtained
\eqref{maineq} and hence the conjecture, we now want to slightly change
Conjecture~\ref{conj1} to
\begin{conjecture} \label{conjnew1} Given integers $a_0\le a_1\le ... \le
  a_{t+c-2}$ and $b_1\le ...\le b_t$, we assume $a_{i-\min ([c/2]+1,t)} \ge
  b_{i}$ provided $n>c$ and $a_{i-\min ([c/2]+1,t)} >
  b_{i}$ provided $n=c$
  for $\min ([c/2]+1,t)\le i \le t$. Except for the family
  $W({0,0};{1,1,...,1})$ of \ {\rm zero-dimensional} schemes above we have,
  for $W(\underline{b};\underline{a}) \ne \emptyset$, that
 \[
 \dim W(\underline{b};\underline{a}) = \lambda_c+ K_3 + K_4+...+K_c \ . \]
\end{conjecture}
Indeed in the situation of Proposition~\ref{main1} we even {\it expect}
$\eqref{maineq}$
to hold! This will imply Conjecture~\ref{conjnew1} provided the conjecture
holds for $W(\underline{b};\underline{a'})$. Note that the conclusion of the
conjecture is true provided $n-c\ge 1$ and $2 \le c \le 5$ ($char(k) = 0$ if
$c=5$) by \cite{KM}.

Finally we will state a conjecture related to the problems (2) and
(3) of the Introduction:
\begin{conjecture} \label{conjcomp} Given integers $a_0\le a_1\le ... \le
  a_{t+c-2}$ and $b_1\le ...\le b_t$, we suppose $n-c\ge 2$, $c \ge 5$ and
   $ \ a_{0} > b_{t}$. 
Then $\overline{
    W(\underline{b};\underline{a})}$ is a generically smooth irreducible
  component of the Hilbert scheme $\Hi ^p(\PP^{n})$.
\end{conjecture}
Indeed due to the results of this paper and many examples computed by Macaulay
2 in the range $ \ a_{0} > b_{t} $ we even {\it expect} the groups $ _0\! \Ext
^1_{D_3}(I_{D_3}/I^2_{D_3},I_{i}) \mbox{ for } i=4,...,c-1 \ $ of
Theorem~\ref{codcomp} to vanish! This will imply Conjecture~\ref{conjcomp}.
The conclusion of the conjecture may even be true for $0 \le n-c \le 1$ (we
have no counterexample), but in this range we have verified that the
$\Ext^1$-groups above do not always vanish. Note that the conclusion of
Conjecture~\ref{conjcomp} is true provided $n-c\ge 2$ and $2 \le c \le 4$
(\cite{elli}, \cite{KMMNP}, \cite{KM}).


\end{document}